\documentclass[12pt]{article}

\usepackage{amsmath, amsthm, amscd, amsfonts, amssymb, graphicx}
\usepackage{microtype}
\usepackage[T1]{fontenc}
\usepackage{color}
\usepackage{tikz}
\usepackage{url}
\usepackage[none]{hyphenat}
\newtheorem{theorem}{Theorem}[section]
\newtheorem{definition}[theorem]{Definition}
\newtheorem{observation}[theorem]{Observation}
\newtheorem{lemma}[theorem]{Lemma}

\newtheorem{corollary}[theorem]{Corollary}

\def\ni{\noindent}

\newcommand{\diam}{\mathop{\mathrm{diam}}}

\title{Mutual-visibility coloring of graphs}
\author{Saneesh Babu $^{a,}$\footnote{\tt saneeshbabu1996@gmail.com} 
   \and Gabriele Di Stefano $^{b,}$\footnote{\tt gabriele.distefano@univaq.it}
	\and Aparna Lakshmanan S $^{a,}$\footnote{\tt aparnaren@gmail.com, aparnals@cusat.ac.in} \\\\
	$^{a}$ \small Department of Mathematics, Cochin University of Science and Technology, \\ 
    \small Cochin - 22, Kerala, India \\
    $^b$ \small Department of Information Engineering, Computer Science and Mathematics, \\
    \small University of L'Aquila, Italy \\
}
\date{}
\begin{document}

\maketitle
\begin{abstract}
The mutual-visibility chromatic number of a graph $G$ is the smallest number of colors needed to color the vertices of $G$ such that each color class is a mutual-visibility set. In this paper, we prove that determining the mutual-visibility chromatic number of a graph is NP-complete even when restricted to the class of graphs having diameter four and mutual-visibility chromatic number two. We further determine the exact value of the mutual-visibility chromatic number for glued binary trees and glued $t$-ary trees.
\end{abstract}

\medskip\noindent
{\bf AMS Subj.\ Class.\ (2020)}: 05C15, 68Q17, 05C69
\section{Introduction}

The \emph{no-three-in-line problem} is a classical problem in discrete geometry known from games in various cultures since ancient times \cite{ANC2,ANC1}. It has been a subject of mathematical study in combinatorics after Dudeney's chessboard puzzle which was proposed in the early twentieth century \cite{DDC}. This problem was later generalized to larger grids and planes \cite{ENT3, ENT, ENT4, ENT2}. Graph theoretic version of this problem was introduced independently in $2016$ and $2018$ under the names \emph{geodesic irredundant sets} \cite{GPS1} and \emph{general position sets} \cite{GPS2}, respectively. Later, the graph theoretic version of the problem is popular under the name general position problem which has been recently surveyed in~\cite{survey}, see also~\cite{Araujo-2025, Klavzar-2025b, KorzeVesel-2025, Roy-2025}. Subsequently, several variations of this problem emerged \cite{Roy-2026}, among which the most extensively studied is the \emph{mutual-visibility}  variant, first introduced by Di Stefano in $2022$ \cite{MVN1}. A set $S \subseteq V(G)$ is mutually visible if for every $x,y \in S$, there exists a shortest $x-y$ path with no internal vertices in $S$. Beyond its application in computer science, particularly in robot visibility, the mutual-visibility problem connects to combinatorial problems such as Zarankiewicz and Turan problems \cite{MVN2,MVN4,MVN3}.\par
In $2024$, Sandi Klav\v zar et al. introduced a vertex coloring of graphs in which each color class is a mutual-visibility set \cite{MVC1}. The \emph{mutual-visibility coloring}  is defined as follows: two vertices $x,y \in V (G)$ can be given the same color if there exists an $x-y$ geodesic path whose internal vertices have colors different from that of $x$ and $y$. The smallest number of colors required for such a coloring of $G$ is called \emph{mutual-visibility chromatic number}, denoted by $\chi_{\mu}(G)$. Following this, vertex coloring based on general position sets \cite{MVC3}, independent mutual-visibility sets \cite{MVC2} and  distance mutual-visibility sets \cite{Babu-2026} have also been investigated.\par
It is proved in \cite{MVC3} that the general position problem is NP-complete even when restricted to the class of graphs with diameter two and general position number three. The concept of independent general position coloring is also introduced and proved to be NP-complete even when restricted to the class of graphs with diameter at most three. In \cite{MVC1}, computing the computational complexity of mutual-visibility chromatic number is left as an interesting problem for further investigation. The decision problem arising from the independent version of mutual-visibility coloring \cite{MVC2} is proved to be NP-complete even in graphs with universal vertices and the authors conjectured that mutual-visibility coloring is NP-complete.\par

A {\it perfect $t$-ary tree}, $T_{r,t}$ of depth $r \geq 1$, is a rooted tree in which all non-leaf vertices have $t$ children, and all leaves have depth $r$. It has $\frac{t^{r+1}-1}{t-1}$ vertices and $t^r$ leaves. For $t \geq 2$ and $r \geq 1$, a {\it glued $t$-ary tree}, $GT(r,t)$ is obtained from two copies of $T_{r,t}$ by pairwise identifying their leaves \cite{Klavzar-2025b}. The vertices obtained by identification are called {\it quasi-leaves} of $GT(r,t)$ and the set of quasi-leaves of $GT(r,t)$ will be denoted by $L(GT(r,t))$. For a vertex $u$ in $GT(r,t)$, $T_u $ denotes the tree rooted in $u$ and having as leaves all the descendants of $u$ that are quasi-leaves.\par

The perfect $2$-ary and glued $2$-ary trees are called perfect binary trees and glued binary trees \cite{GPS2}, respectively. We will simplify the notation by setting $GT(r)=GT(r,2)$. We will denote the two copies of $T_{r,2}$, from which $GT(r)$ is constructed, by $T_r^{(1)}$ and $T_r^{(2)}$, so that $L(GT(r))= V(T_r^{(1)}) \cap V(T_r^{(2)})$. Moreover, we set $V_r^{(1)} = V(T_1^{(1)}) \backslash L(GT(r))$ and $V_r^{(2)} = V(T_1^{(2)}) \backslash L(GT(r))$.\par

In this paper, the mutual-visibility coloring problem is proved to be NP-complete settling the conjecture. In view of its computational complexity, it is natural to study this parameter in specific graph classes. In the last section we determine the exact values of mutual-visibility coloring for glued binary trees and extend it further the glued $t$-ary trees.
	
\section{Computational complexity}\label{sec:complexity}
To study the computational complexity of finding a mutual-visibility coloring of a graph that uses the smallest possible number of colors, we introduce the following decision problem.
\begin{definition}
{\sc MV-coloring} problem: \\
{\sc Instance}: A graph $G$, a positive integer $k\leq |V(G)|$. \\
{\sc Question}: Is there a mutual-visibility coloring of $G$ such that $\chi_\mu(G)\leq k$?
\end{definition}
The problem is hard to solve as shown by the next Theorem~\ref{thm:mvc-NP}. Before proving the theorem, we introduce a graph whose properties are widely used in the proof of the theorem.
Let $H_n$, $n\geq 2$, be the graph consisting of two stars on $n+2$ vertices where $n$ leaves are identified. Let $c$ and $p$ denote the center and the non-identified leaf of the first star. Let $c'$ and $p'$ denote the corresponding vertices of the second star.
Let $p_1, p_2,\ldots,p_n$ be the identified leaves. A representation of $H_n$ is given in Figure~\ref{fig:Hn}.

\begin{figure}
    \centering
    %% Creator: Inkscape 1.4.2 (2aeb623e1d, 2025-05-12), www.inkscape.org
%% Accompanies image file 'Hn.pdf' (pdf, eps, ps)
\begingroup%
  \makeatletter%
  \providecommand\color[2][]{%
    \errmessage{(Inkscape) Color is used for the text in Inkscape, but the package 'color.sty' is not loaded}%
    \renewcommand\color[2][]{}%
  }%
  \providecommand\transparent[1]{%
    \errmessage{(Inkscape) Transparency is used (non-zero) for the text in Inkscape, but the package 'transparent.sty' is not loaded}%
    \renewcommand\transparent[1]{}%
  }%
  \providecommand\rotatebox[2]{#2}%
  \newcommand*\fsize{\dimexpr\f@size pt\relax}%
  \newcommand*\lineheight[1]{\fontsize{\fsize}{#1\fsize}\selectfont}%
  \ifx\svgwidth\undefined%
    \setlength{\unitlength}{100.38020325bp}%
    \ifx\svgscale\undefined%
      \relax%
    \else%
      \setlength{\unitlength}{\unitlength * \real{\svgscale}}%
    \fi%
  \else%
    \setlength{\unitlength}{\svgwidth}%
  \fi%
  \global\let\svgwidth\undefined%
  \global\let\svgscale\undefined%
  \makeatother%
  \begin{picture}(1,1.02891765)%
    \lineheight{1}%
    \setlength\tabcolsep{0pt}%
    \put(0,0){\includegraphics[width=\unitlength,page=1]{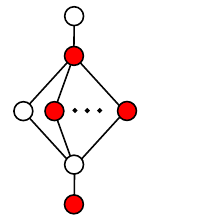}}%
    \put(0.41926528,0.19496974){\color[rgb]{0,0,0}\makebox(0,0)[lt]{\lineheight{1.25}\smash{\begin{tabular}[t]{l}$c'$\end{tabular}}}}%
    \put(0.42673687,0.78735466){\color[rgb]{0,0,0}\makebox(0,0)[lt]{\lineheight{1.25}\smash{\begin{tabular}[t]{l}$c$\end{tabular}}}}%
    \put(0.42673687,0.96667289){\color[rgb]{0,0,0}\makebox(0,0)[lt]{\lineheight{1.25}\smash{\begin{tabular}[t]{l}$p$\end{tabular}}}}%
    \put(0.41947101,0.01657756){\color[rgb]{0,0,0}\makebox(0,0)[lt]{\lineheight{1.25}\smash{\begin{tabular}[t]{l}$p'$\end{tabular}}}}%
    \put(-0.00661552,0.41377508){\color[rgb]{0,0,0}\makebox(0,0)[lt]{\lineheight{1.25}\smash{\begin{tabular}[t]{l}$p_1$\end{tabular}}}}%
    \put(0.31644423,0.41234252){\color[rgb]{0,0,0}\makebox(0,0)[lt]{\lineheight{1.25}\smash{\begin{tabular}[t]{l}$p_2$\end{tabular}}}}%
    \put(0.64341352,0.41342684){\color[rgb]{0,0,0}\makebox(0,0)[lt]{\lineheight{1.25}\smash{\begin{tabular}[t]{l}$p_n$\end{tabular}}}}%
  \end{picture}%
\endgroup%

    \caption{The graph $H_n$ with a possible two coloring.}
    \label{fig:Hn}
\end{figure}
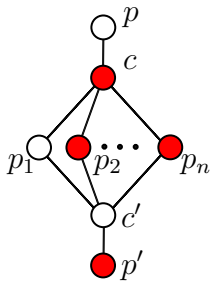

\begin{lemma}\label{lem:Hn}
    In a mutual-visibility coloring of $H_n$ with two colors, vertices $c$ and $p$ ($c'$ and $p'$, respectively) cannot be colored with the same color. In the set of vertices $\{p_1, p_2,\ldots,p_n\}$, at least two vertices are colored with different colors.
\end{lemma}
\begin{proof}
Consider a mutual-visibility coloring of $H_n$ using two colors, say red and white. If a pendant vertex and its support vertex are given the same color, then that color cannot be used anywhere else in the graph. Therefore, if $p$ and $c$ (equivalently, $p'$ and $c'$) are given the same color, then we require at least three colors to color $H_n$.  Suppose vertices $p$ and $c$ are colored red. Therefore, $p$ and $c$ (equivalently, $p'$ and $c'$) must be of different colors. Without loss of generality assume that $p$ is white and $c$ is red. Since at least one among $c'$ and $p'$ is white (equivalently, red), at least one among $\{p_1, p_2, ..., p_n\}$ must be red (equivalently, white) for mutual-visibility. Hence, the lemma. 
\end{proof}

\begin{figure}\label{fig:NP}
    \centering
    \scalebox{0.8} {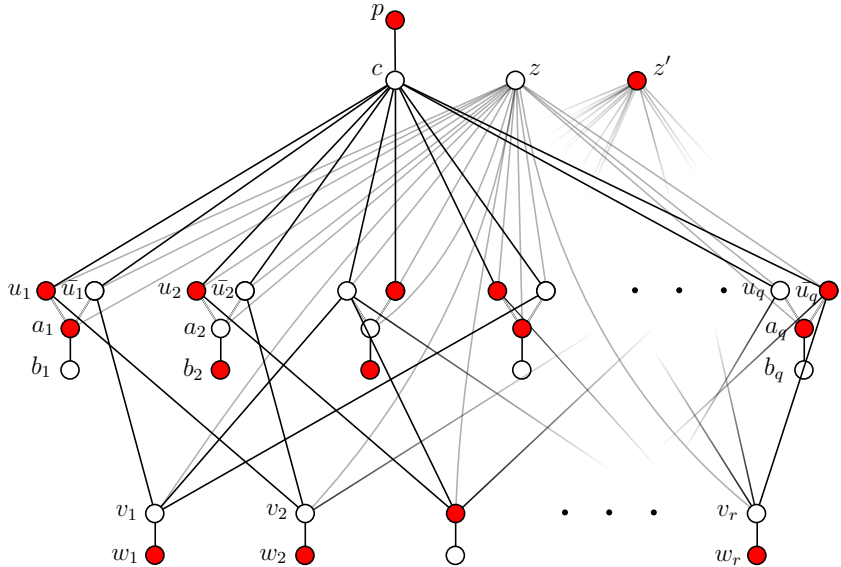}
    \caption{The graph $G$ used in the proof of Theorem~\ref{thm:mvc-NP}. Clause $c_1=\{\bar{x_1},x_3,\bar{x_4}\}$ is represented by the $H_3$ subgraph induced by vertices $p,c,v_1,w_1$ $v_1$ and the three vertices $\bar{u_1},u_3$, and $\bar{u_4}$ in $T_1$.}
    \label{fig:G}
\end{figure}
\begin{theorem}\label{thm:mvc-NP}
 {\sc MV-coloring} is NP-complete even for instances $(G,k)$ with $\diam(G) = 4$ and $k=2$.  
\end{theorem}  

\begin{proof} 
 Given a coloring of $G$, it is possible to test in polynomial time whether it is a mutual-visibility coloring or not. Consequently, the problem is in NP. 
 
 We prove that the {\sc NAE3-SAT} problem, defined below and shown to be NP-complete in~\cite{schafer-1978}, polynomially reduces to {\sc MV-coloring}.
 \begin{quote}  
 A {\sc NAE3-SAT} (\emph{not-all-equal 3-satisfiability}) instance $\Phi$ is defined as a set $X=\{x_1,x_2,\ldots,x_q\}$ of $q$ Boolean variables and a set $C$ of $r$ clauses, each  defined as a set of three literals: every variable $x_i$ corresponds to two literals, that is, $x_i$ (the positive form) and $\bar x_i$ (the negative form). To simplify the notations we will denote by $\{\ell_1, \ell_2, \ell_3\}$ the clause with literals $\ell_i, i\in[3]$, without distinction between the orders in which they are listed. A truth assignment assigns a Boolean value ($True$ or $False$) to each variable, corresponding to a truth assignment of opposite values for the two literals $x_i$ and $\bar {x_i}$: $\bar{x_i}$ is $True$ if and only if $x_i$ is $False$. 
 
 The {\sc NAE3-SAT} problem asks whether there is a truth assignment to the variables such that in no clause all three literals have the same truth value. We will say that such an assignment is \emph{satisfying} and the instance $\Phi $ is \emph{satisfied}.
\end{quote}

In what follows we assume that in any instance of {\sc NAE3-SAT} each clause involves three different variables. Indeed, if a clause involves a single variable and all the literals are equal, then the {\sc NAE3-SAT} problem has a No answer and if the positive and negative form of a variable are present in a clause we can remove the clause from the instance.
If a clause involves only two variables, and the variable appearing two times in the clauses is present in the same form, e.g. $\{l_1,l_1,l_2\}$, then the clause is logically equivalent to $\{l_1,l_2\}$ and it can be substituted with two clauses $\{l_1,l_2, a\}$ and $\{l_1,l_2, \bar a\}$, where $a$ is an additional variable.

\bigskip

We reduce any instance of {\sc NAE3-SAT} to an instance of {\sc MV-Coloring}. Let $\Phi=(X,C)$, with $X = \{x_1,x_2,\ldots,x_q\}$
and $C = \{c_1,c_2,\ldots,c_r\}$, be any instance of {\sc NAE3-SAT}. We construct a graph $G = (V,E)$ in such a way that $G$ has a mutual-visibility coloring with at most $k=2$ colors if and only if $\Phi$ admits a satisfying truth assignment.

For each variable $x_i\in X$, $i\in [q]$, there is a true-setting  subgraph $H_2^{(i)}=(V_i,E_i)$ of $G$ with $V_i=\{u_{i},\bar{u_i},a_i,b_i,p,c\}$ and $E_i=\{u_{i}a_i,\bar{u_i}a_i,u_{i}c,\bar{u_i}c, a_ib_i, pc\}$. Then, each subgraph $H_2^{(i)}$ is isomorphic to a $H_2$ graph. Notice that, by Lemma~\ref{lem:Hn}, the vertices $u_{i}$ and $\bar{u_i}$ of $H_2^{(i)}$ cannot be colored with the same color, as well as the pairs of vertices $a_i,b_i$ and $p,c$.

For each clause $c_j \in C$, $j\in [r]$, there is subgraph $H_3^{(j)}=(V_j,E_j)$ of $G$ isomorphic to a $H_3$ graph where $V_j$ contains the vertices $v_j,w_j,p,c$ and three more vertices depending on $c_j$. Let $T_j$ be the set of these three vertices. If the clause $c_j=\{\ell_{i_1},\ell_{i_2},\ell_{i_3}\}$ contains the variables $x_{i_1},x_{i_2}$, and $x_{i_3}$ then the three vertices of $H_3^{(j)}$ in $T_j$ belong to $H_2^{(i_1)}$, $H_2^{(i_2)}$, and $H_2^{(i_3)}$. In particular if $\ell_{i_1}=x_{i_1}$  ($\ell_{i_1}=\bar x_{i_1}$, resp.)  then $u_{i_1}$ ($\bar u_{i_1}$, resp.) is a vertex of $T_j$.  Similarly for the vertices corresponding to literals $\ell_{i_2}$ and $\ell_{i_3}$. Set $E_j$ contains edges $pc$ and $v_jw_j$ and all the edges connecting  vertices in $\{c, v_j\}$ to vertices in $T_j$. Notice that, by Lemma~\ref{lem:Hn}, vertices in $T_j$ cannot be colored with the same color. The same happens for the pairs of vertices $v_j$ and $w_j$, $j\in[r]$.

There are two more non-adjacent vertices in $V$, that is $z$ and $z'$. These two vertices are adjacent to all the other vertices in the set $\{u_i,\bar{u_i},a_i,v_j~|~i\in[q], j\in[r]\}$. Then

$$ V=\bigcup_{i\in [q]} V_i \cup \bigcup_{j\in [r]}V_j\cup\{z,z'\}$$

$$ E=\bigcup_{i\in [q]} E_i \cup \bigcup_{j\in[r]} E_j\cup\left\{u_{i}\zeta,\bar u_{i}\zeta,a_{i}\zeta,v_i\zeta~|~i\in[q],j\in[r],\zeta\in\{z,z'\}\right\}$$

\bigskip
A representation of the graph $G$ is given in Figure~\ref{fig:NP}.

\bigskip
It is easy to see that the construction of $G$ can be accomplished in polynomial time and that $\diam(G)=4$. All that remains to be proved is that $\Phi$ is satisfied if and only if $G$ has a mutual-visibility coloring of size at most two.

\bigskip
$(\Rightarrow)$ First, suppose that $t: X\rightarrow \{True,False\}$ is a satisfying truth assignment for $C$. 
Then we color the vertices in $G$ with two colors, namely red, and white, depending on $t$. We color the vertices of the following pairs with different colors: $p,c$; $z,z'$; $a_i,b_i$, for each $i\in[q]$; $v_j,w_j$, for each $j\in [r]$.
If $t(x_i)$ is $True$ for a given $i\in [p]$, then  the vertex $u_{i}$ in subgraph $H_2^{(i)}$ is colored red and vertex $\bar{u_i}$ is colored white. Conversely, if $t(x_i)$ is $False$ then  vertex $u_{i}$ is colored white and vertex $\bar {u_i}$ is colored red.
It remains to show that this coloring is a mutual-visibility coloring. \\
Vertices $p$ and $c$ are in mutual-visibility with all other vertices in $H_2^{(i)}$, for each $i\in [q]$, because the color assignment satisfies that of the proof of Lemma~\ref{lem:Hn}. Clearly, vertices $p$ and $c$ are in mutual-visibility with vertex $z$ ($z'$, resp.) through path $p,c,u_1,z$ ($p,c,u_1,z'$, resp.) or path $p,c,\bar{u_1},z$ ($p,c,\bar{u_1},z'$, resp.).\\
More interesting is the mutual-visibility of $p$ and $c$ with vertices $v_j,w_j$, for each fixed $j\in [r]$. These four vertices belong to $H_3^{(j)}$ and since the other three vertices in $T_j$ cannot have the same color because $t$ is a satisfying truth assignment for $C$, the coloring of $H_3^{(j)}$ also satisfies that of the proof of Lemma~\ref{lem:Hn}. Then $p$ and $c$ are in mutual-visibility with vertices $v_j,w_j$.\\
It remains to check the mutual-visibility of the other vertices. Vertex $z$ is clearly in mutual-visibility with each vertex $s$ in $\bigcup_{i\in [q]} V(H_2^{(i)})\cup \bigcup_{ j\in [r]} V(H_3^{(j)})\setminus \{p,c\}$, since any shortest $x,s$-path $P$ has length at most two and and when the length is two at least two vertices in $P$ have different colors. Similarly for $z'$.\\
 Any vertex $x$ in $A=\{u_i, \bar{u_i}~|~i \in [q]\}$ is in mutual-visibility with any vertex $y$ in $B=\{u_i,\bar{u_i},a_i,b_i~|~i \in [q]\}\cup\{v_j,w_j~|~j\in[r]\}$, because if $d(x,y)=2$ we can choose a shortest $x,y$-path $P$ passing through either $z$ or $z'$ to guarantee a two coloring of $P$, and if  $d(x,y)=3$, that is $y=b_i$ or $y=w_j$, for any $i\in [q]$ or any $j\in [r]$, there is a shortest $x,y$-path passing throw either $z$ or $z'$ such that the first two vertices and the last two vertices are colored with different colors.\\
 Finally, any vertex $x$ in $\{a_i, b_i,v_j,w_j~|~i \in [q], j\in[r]\}$ is in mutual-visibility with any other vertex $y$ in the same set because the subgraph $H$ induced by the vertices $\{a_i,b_i,z,z',v_j,w_j\}$ is isomorphic to $H_2$ and the distances between two vertices in $H$ are the same as those in $G$. Since the coloring of the vertices in $H$ is the coloring used in Lemma~\ref{lem:Hn}, $x$ and $y$ are mutually visible.

 \bigskip
 $(\Leftarrow)$ Conversely, suppose that there is a mutual-visibility coloring of the graph $G$ with two colors. Since the subgraphs $H_2^{(i)}$, for each $i\in[q]$, and $H_3^{(j)}$, for each $j\in [r]$, of $G$ are isomorphic
to the graphs $H_2$ and $H_3$, by Lemma~\ref{lem:Hn} all pairs of vertices $p,c$; $u_i,\bar{u_i}$ and $a_i,b_i$,  for each $i\in[q]$; $v_j,w_j$, for each $j\in [r]$, must be colored with different colors. The pair $z,z'$ must also be colored with two colors by the same lemma, because the subgraph induced by vertices $a_1,b_1, z,z',v_1,w_1$ is isomorphic to $H_2$. 

Since the vertices $u_i$ and $\bar{u_i}$ of a $H_2^{(i)}$ subgraph, for each $i\in[q]$, have different colors, we assign the value $True$ to variable $x_i$ if the vertex $u_{i}$ is colored red, whereas we assign the value $False$ to variable $x_i$ if the vertices $u_{i}$  is colored white. Now, by contradiction, assume that the instance $\Phi$ of {\sc NAE3-SAT} is not satisfied. This means that there exists a clause $c_h\in C$, for a certain $h\in[r]$, such that all three literals of $c_h$ have the same truth value. In turn, this implies that all the vertices in $T_h$ have the same color, but this contradicts Lemma~\ref{lem:Hn} applied to subgraph $H_3^{(h)}$. In particular, there would be a vertex in $x\in\{p,c\}$ and a vertex $y\in \{v_h,w_h\}$ that are not in mutual-visibility, since they would have the same color as the vertices in $T_h$.
\end{proof}

\section{Mutual-visibility coloring in glued $t$-ary trees}

As the mutual-visibility coloring problem is NP-complete even for instances $(G, k)$ with $diam(G) = 4$ and $k = 2$, it is worth investigating the problem in special classes of graphs. The exact value of mutual-visibility chromatic number of the glued $t$-ary trees are determined in this section. We start with the special case of the glued binary trees and then extend the result to glued $t$-ary trees. \par
Throughout this section, the vertices of two copies of the binary tree in $GT(r)$ are denoted by $v_{i,j}$ and $v_{i,j}'$, respectively, for $i = 1, 2, \ldots, r - 1$ and $i-1$ is the depth of $v_{i,j}$ and $v'_{i,j}$ in $T_r^{(1)}$ and $T_r^{(2)}$, respectively; and $j$ is the position of the vertex counted from left to right in the drawing.  The quasi-leaves, which are leaf vertices of both rooted binary trees, are denoted by $v_1,v_2,......v_{2^r}$, labeled from left to right in the drawing.  See Figure~\ref{fig:mvc} for the denotation of the vertices in graph $GT(3)$ and graph $GT(4)$.
%\textcolor{blue}{To facilitate reading, the denotation of the vertices (e.g. $v_{1,1}$, $v_{2,1}, \ldots$) could be added to Figure~\ref{mvc} or a new figure; Also, I propose to use  $T_r$ and $T'_r$ instead of $T_r^{(1)}$ and $T_r^{(2)}$, (and $V_r$ and $V'_r$ instead of $V_r^{(1)}$ and $V_r^{(2)}$) according to the denotation of the vertices and the paths in the corresponding trees.}\par
%\textcolor{green}{The notation $T_r^{(1)}$ and $T_r^{(2)}$ were used keeping in mind the possibility of extending the results to generalized glued binary trees. These notations were used in our previous paper with Sandi and Dhanya. If you are not very particular to change it, we prefer to keep this notation for consistency.}

Let $v_i$ and $v_j$ be two quasi-leaves of $GT(r)$. There exist exactly two geodesic paths between $v_i$ and $v_j$, say $P_{i,j}$, through $V_r^{(1)}$ and $P'_{i,j}$ through $V_r^{(2)}$. Now, $P_{i,j}$ and $P'_{i,j}$ together induces a cycle which we denote by $C_{i,j}$. Let $Q_{i,j}$ and $Q_{i,j}'$ be the paths $P_{i,j} \backslash \{v_i,v_j\}$ and paths $P_{i,j}' \backslash \{v_i,v_j\}$, respectively. Note that any cycle in $GT(r)$ intersects with exactly two quasi-leaves and can always be represented as $C_{i,j}$ in a unique manner.\par

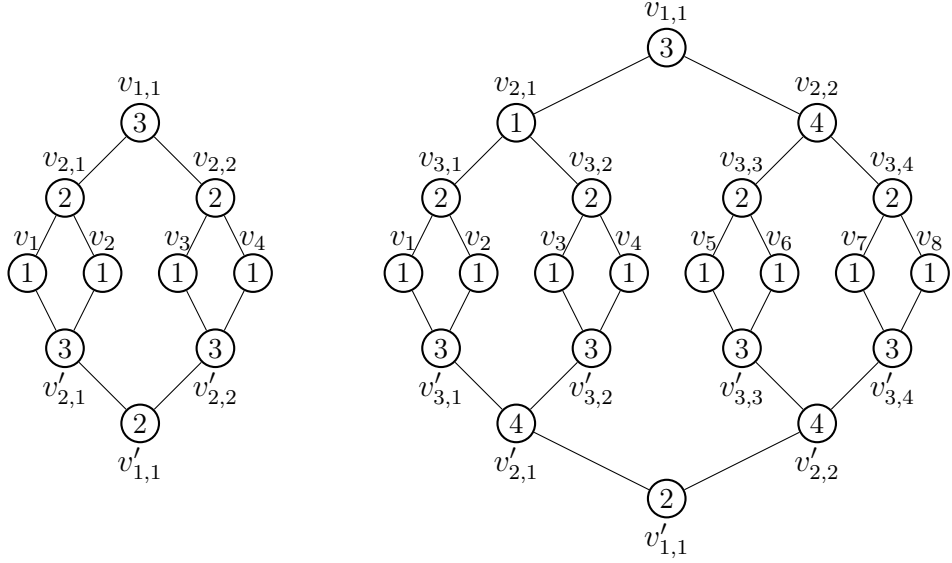
\begin{figure}[thb]
    \centering
   
			\begin{tikzpicture}[vertex/.style= {circle,draw,thick,minimum size=5mm,inner sep=0pt,font=\small},x=1cm,y=1cm]

		\node[vertex, label={[yshift=-1mm]:$v_1$}] (a1) at (0,0) {1}; 
       
	\node[vertex,label={[yshift=-1mm]:$v_2$}] (a2) at (1,0) {1};  
	\node[vertex,label={[yshift=-1mm]:$v_3$}] (a3) at (2,0) {1};  
	\node[vertex,label={[yshift=-1mm]:$v_4$}] (a4) at (3,0) {1};

		\node[vertex,label={[yshift=-1mm]:$v_{2,1}$}] (b1) at (0.5,1) {2}; 
	\node[vertex,label={[yshift=-1mm]:$v_{2,2}$}] (b2) at (2.5,1) {2};  
		\node[vertex,label={[yshift=-1mm]:$v_{1,1}$}] (d1) at (1.5,2) {3}; 
		\node[vertex,label={[yshift=-12mm]:$v_{2,1}'$}] (c1) at (0.5,-1) {3}; 
			\node[vertex,label={[yshift=-12mm]:$v_{2,2}'$}] (c2) at (2.5,-1) {3};  
			\node[vertex,label={[yshift=-12mm]:$v_{1,1}'$}] (e1) at (1.5,-2) {2}; 
		
			\draw (a1) -- (b1) -- (d1)--(b2) -- (a4) -- (c2)--(e1) -- (c1)--(a1);
			\draw (c1)--(a2)--(b1);
			\draw (c2)--(a3)--(b2);

			\node[vertex,label={[yshift=-1mm]:$v_1$}] (a1) at (5,0) {1}; 
			\node[vertex,label={[yshift=-1mm]:$v_2$}] (a2) at (6,0) {1};  
			\node[vertex,label={[yshift=-1mm]:$v_3$}] (a3) at (7,0) {1};  
		 	\node[vertex,label={[yshift=-1mm]:$v_4$}] (a4) at (8,0) {1}; 
	    	\node[vertex,label={[yshift=-1mm]:$v_5$}] (a5) at (9,0) {1};  
		    \node[vertex,label={[yshift=-1mm]:$v_6$}] (a6) at (10,0) {1}; 
	    	\node[vertex,label={[yshift=-1mm]:$v_7$}] (a7) at (11,0) {1}; 
			\node[vertex,label={[yshift=-1mm]:$v_8$}] (a8) at (12,0) {1};

			\node[vertex,label={[yshift=-1mm]:$v_{3,1}$}] (b1) at (5.5,1) {2}; 
			\node[vertex,label={[yshift=-1mm]:$v_{3,2}$}] (b2) at (7.5,1) {2};  
		\node[vertex,label={[yshift=-1mm]:$v_{3,3}$}] (b3) at (9.5,1) {2};  
		\node[vertex,label={[yshift=-1mm]:$v_{3,4}$}] (b4) at (11.5,1) {2}; 
		
	\node[vertex,label={[yshift=-1mm]:$v_{2,1}$}] (d1) at (6.5,2) {1}; 
\node[vertex,label={[yshift=-1mm]:$v_{2,2}$}] (d2) at (10.5,2) {4}; 

\node[vertex,label={[yshift=-1mm]:$v_{1,1}$}] (f1) at (8.5,3) {3}; 
	
		\node[vertex,label={[yshift=-12mm]:$v_{3,1}'$}] (c1) at (5.5,-1) {3}; 
	\node[vertex,label={[yshift=-12mm]:$v_{3,2}'$}] (c2) at (7.5,-1) {3};  
	\node[vertex,label={[yshift=-12mm]:$v_{3,3}'$}] (c3) at (9.5,-1) {3};  
	\node[vertex,label={[yshift=-12mm]:$v_{3,4}'$}] (c4) at (11.5,-1) {3}; 
	
	\node[vertex,label={[yshift=-12mm]:$v_{2,1}'$}] (e1) at (6.5,-2) {4}; 
	\node[vertex,label={[yshift=-12mm]:$v_{2,2}'$}] (e2) at (10.5,-2) {4}; 
	
	\node[vertex,label={[yshift=-12mm]:$v_{1,1}'$}] (g1) at (8.5,-3) {2};

			\draw (a1) -- (b1) -- (d1)--(f1) -- (d2) -- (b4)--(a8) -- (c4) -- (e2)--(g1) -- (e1) -- (c1)--(a1) ;
			\draw (e1) -- (c2) -- (a3)--(b2) -- (a4) -- (c2);
			\draw (e2) -- (c3) -- (a5)--(b3) -- (a6) -- (c3);
			\draw (d1) -- (b2);
			\draw (d2) -- (b3);
			\draw (b4) -- (a7) -- (c4);
				\draw (b1) -- (a2) -- (c1);

		\end{tikzpicture}

    \caption{A mutual-visibility coloring of $GT(2)$ with $3$ colors and $GT(3)$ with $4$ colors. }
  \label{fig:mvc}
\end{figure}
The following observations related to the geodesic paths between any two vertices in $GT(r)$ will be useful for us.

\begin{observation}
Given two vertices of the form $v_{i,j}$ and $v_{k,l}$ in a copy of a binary tree, there exists a unique geodesic path between them.
\end{observation}

\begin{observation}
Given two vertices $v_{i,j}\in T^{(1)}_r$, $v'_{i,j}\in T^{(2)}_r$ and any leaf vertex $v_a$ of  $T_{v_{i,j}}$ (equivalently, $T_{v'_{i,j}})$, then there exists a unique geodesic path between $v_{i,j}$ and $v'_{i,j}$, which passes through $v_a$.
\end{observation}

\begin{observation}
    Given two vertices $v_{i,j}\in T^{(1)}_r$, $v'_{i,l}\in T^{(2)}_r$ and any leaf vertex $v_a$ of  $T_{v_{i,j}}$ (equivalently, $T_{v'_{i,l}})$, then there exists a unique geodesic path between $v_{i,j}$ and $v'_{i,l}$, which passes through $v_a$. Also each of these geodesic paths, either passes through $v'_{i,j}$ or $v_{i,l}$.
\end{observation}
\begin{observation}
    Given two vertices $v_{i,j}\in T^{(1)}_r$, $v'_{k,l}\in T^{(2)}_r$ where $i<k$ and any leaf vertex $v_a$ of  $T_{v_{k,l}}$, then there exists a unique geodesic path between $v_{i,j}$ and $v'_{k,l}$, which passes through $v_a$ and each of these geodesic paths passes through $v_{k,l}$ also. Also, a similar statement holds for $k < i$.
\end{observation}

\begin{lemma} \label{lem}
If $S$ is a mutual-visibility set of $GT(r)$, then $|{S \cap V(C_{i,j}) }| \leq 3 $.
\end{lemma}

\begin{proof}
Assume the contrary that there exists a mutual-visibility set $S$ of $GT(r)$ such that $|S \cap V(C_{i,j})| > 3$.\par

If both quasi-leaves in $C_{i,j}$ are present in $S$, then $|S \cap V(Q_{i,j})| \leq 1$ and $|S \cap V(Q_{i,j}')| \leq 1$. Also, $S$ cannot simultaneously intersect $V(Q_{i,j})$ and $V(Q_{i,j}')$ and hence $|S \cap V(C_{i,j})| \leq 3$. If one of the quasi-leaves in $C_{i,j}$ is present in $S$, then again $|S \cap V(Q_{i,j})| \leq 1$ and $|S \cap V(Q_{i,j}')| \leq 1$, and hence $|{S \cap V(C_{i,j}) }| \leq 3$.\par

Now consider the case where both the quasi-leaves are not present in $S$. Since the quasi-leaves are not present in $S$, then since $|S \cap V(Q_{i,j})| \leq 2$ and $|S \cap V(Q_{i,j}')| \leq 2$, we have $|{S \cap V(C_{i,j}) }| \leq 4$. But, if there are two vertices each from both $Q_{i,j}$ and $Q'_{i,j}$, say $v_{i,j}, v_{k,l}, v'_{a,b}$ and $v'_{c,d}$, then from the above observations, at least one pair of these vertices will not be mutually-visible, which is a contradiction. Hence, the lemma.
\end{proof}
We now turn our focus to determining the exact  value of $\chi_{\mu}(GT(r))$. Observe that $\chi_{\mu}(GT(1))=\chi_{\mu}(C_4)=2$.

\begin{theorem} \label{binary}
For $r \geq 2$,
	$$\chi_\mu(GT(r))=\begin{cases}
        2(r-i) + 3, & 2^{i-1}+i-2 \leq r \leq 2^{i-1}+2^{i-2}+i-3\\
	    2(r-i) + 2, & 2^{i-1}+2^{i-2}+i-2 \leq r \leq 2^i + i - 2       
	\end{cases}$$
\end{theorem} 
    
 \begin{proof}
We begin the proof by establishing a coloring of $GT(r)$ using the number of colors specified in the theorem.\par

\ni{\bf Case 1: $2^{i-1}+2^{i-2}+i-2 \leq r \leq 2^i + i - 2$}.\\
In this case, assign color $1$ to all the quasi-leaves, color $2k$ to all the vertices of depth $r-k$ in $V_r^{(1)}$ and to the $k^{th}$ vertex in the sequence $\{v'_{1,1}, v'_{2,1},v'_{2,2}, v'_{3,1},$ $v'_{3,2}, v'_{3,3}, v'_{3,4},\ldots\}$, for $k = 1, 2, \ldots, r-i$. Similarly, assign color $2k+1$ to all vertices of depth $r-k$ in $V_r^{(2)}$ and to the $k^{th}$ vertex in the sequence $\{v_{1,1}, v_{2,2},v_{2,1}, v_{3,4}, v_{3,3}, v_{3,2}, v_{3,1},\ldots\}$, for $k = 1, 2, \ldots, r-i$. Here all vertices of depth $0,1,\ldots,i-2$ and $r-1,r-2,\ldots,r-(r-i)=i$ are completely colored and vertices of depth $i-1$ are partially colored. Among the vertices of depth $i-1$ in $V_r^{(2)}$, $v'_{i,1},v'_{i,2},\ldots,v'_{i,j}$ will be already colored, where $j=r-i-(2^{i-1}-1) \geq 2^{i-2}-1$, (since, $r \geq 2^{i-1}+2^{i-2}+i-2$). Similarly, in $V_r^{(1)}$, $v_{i,a},v_{i,a-1},\ldots,v_{i,a-j-1}$ will be already colored, where $a=2^{i-1}$ and $j \geq 2^{i-2} -  1$. When $r \geq 2^{i-1}+2^{i-2}+i - 1$, we have $j \geq 2^{i-2}$ so that there are fewer than $2^{i-2}$ uncolored vertices of depth $i-1$. In this case, color all the remaining uncolored vertices with color $2(r-i)+2$. When $r = 2^{i-1}+2^{i-2}+i - 2$, color $v_{i-1,j+1}$ with color 1 and all the remaining vertices with color $2(r-i)+2$ to obtain a mutual-visibility coloring of $GT(r)$ with $2(r-i)+2$ colors.\\
\ni{\bf Case 2: $2^{i-1}+i-2 \leq r \leq 2^{i-1}+2^{i-2}+i-3$}.\\
\ni{\bf Sub Case 1: $r=2^{i-1}+i-2$}.\\
In this case, we proceed with the same coloring pattern given in Case 1 up to $k=r-i+1$ so that the entire vertices of $GT(r)$ are colored using  $2(r-i)+3$ colors.\\
\ni{\bf Sub case 2: $2^{i-1}+i-1 \leq r \leq 2^{i-1}+2^{i-2}+i-3$}.\\
The coloring procedure is the same as Case 1, where the only difference is that more than half of the vertices of depth $i-1$ (all vertices when $r=2^{i-1}+i-1$) will be uncolored. We can use one color each for these uncolored vertices in $V_r^{(1)}$ and $V_r^{(2)}$ so that the resulting coloring is a mutual-visibility coloring of $GT(r)$ by Lemma \ref{lem}. 

To prove the converse, first we prove that $2(r-i)+1$ colors are not enough for a mutual-visibility coloring of $GT(r)$, where $i$ is the unique integer such that $2^{i-1} + i-2 \leq r \leq 2^i+i-2$. Assume the contrary that there exists a mutual-visibility coloring with $2(r-i)+1$ colors.\par

Consider the cycle $C_{1,2^r}$. By Lemma \ref{lem}, a color can be used at most 3 times to color the vertices of this cycle. Since we have only $2(r-i)+1$ colors, and there are $4r$ vertices in $C_{1,2^r}$, at least $4r - 2[2(r-i)+1] = 4i-2$ colors must be used at least three times. Let $\chi_0$ be the collection of colors which are used three times to color the vertices of $C_{1,2^r}$ so that $|\chi_0|\geq4i-2$. Now, consider the cycle $C_{2^{r-1},2^{r-1}+1}$, which is also a cycle of length $4r$. Note that, $|V(C_{1,2^r}) \cap V(C_{2^{r-1},2^{r-1}+1})| = 6$ and there are $4r-6$ vertices in $V(C_{2^{r-1},2^{r-1}+1}) \setminus V(C_{1,2^r}) $. Similar to the discussion above, at least $4i - 8$ colors  must be used three times to color these vertices. Let $\chi_1$  be the collection of colors which are used three times to color the vertices in $V(C_{2^{r-1},2^{r-1}+1}) \setminus V(C_{1,2^r})$ so that $|\chi_1|\geq4i - 8$. We claim that $\chi_0  \cap \chi_1  = \phi$. On the contrary, if $c \in \chi_0  \cap \chi_1$ then $c$ must be used at least two times in the geodesic path $v_{1,1}-v_1-v'_{1,1}$ or in the geodesic path $v_{1,1}-v_{2^r}-v'_{1,1}$. Similarly, $c$ must be used at least two times in the geodesic $v_{3,2}-v_{2^{r-1}}-v'_{3,2}$ or $v_{3,3}-v_{2^{r-1}+1}-v'_{3,3}$. But then, in one of the cycles $C_{1,2^{r-1}}, C_{1,2^{r-1}+1}, C_{2^{r-1},2^r}$ or $C_{2^{r-1}+1,2^r}$, $c$ will be used four times which contradicts Lemma \ref{lem}. Therefore,$\chi_0  \cap \chi_1  = \phi$.

Now, consider the cycles 
%$C_{2^{r-2},2^{r-2}+1}$ and $C_{2^{r-1}+2^{r-2},2^{r-1}+2^{r-2}+1}$
$C_{t2^{r-2},t2^{r-2}+1}$, for $t=1,3$. These two cycles have $4r-12$ vertices each which are not part of both the previous cycles considered. Similar to the case above we can find a set of colors $\chi_2$ which are used three times either in $V(C_{2^{r-2},2^{r-2}+1}) \setminus \{ V (C_{1,2^r})\cup V( C_{2^{r-1},2^{r-1}+1})\}$ and $V(C_{2^{r-1}+2^{r-2},2^{r-1}+2^{r-2}+1}) \setminus \{V(C_{1,2^r}) \cup V(C_{2^{r-1},2^{r-1}+1})\}$. Also, $ \chi_2 \cap \chi_1 = \chi_2 \cap \chi_0 = \phi$ and $|\chi_2| \geq 2(4i - 12)$. Now, we consider the four cycles 
%$C_{t+2^{r-3},t+2^{r-3}+1}$, where $t = 0, 2^{r-2}, 2^{r-1}, 2^{r-1}+2^{r-2}$ 
$C_{t2^{r-3},t2^{r-3}+1}$, for $t=1,3,5,7$, and proceed as above to get $\chi_3$ of cardinality at least $4(4i - 16)$ which do not intersect pairwise with any of $\chi_0,\chi_1$ and $\chi_2$. Continue like this to get a collection of distinct colors $\chi_1, \chi_2,\ldots,\chi_{i-1}$, where each $\chi_j$ is of cardinality at least $2^{i-1}(4i - 4j - 4)$, $j \in [i-1]$ and $\chi_0$ of cardinality at least $4i-2$. Note that the disjoint sets $\chi_j$, $j \in [i-1]$ have cardinality at least $2^{i-1}(4i - 4j - 4)$ which is an arithmetic-geometric progression with common difference -4 and common ratio 2. Therefore, $|\cup_{j=1}^{i-1}\chi_j |\geq 2^{i+1}-4i$ and together with the $4i-2$ more distinct colors in $\chi_0$, at least $2^{i+1}-2$ colors are used three times to color $GT(r)$. But, we have only $2(r-i)+1$ colors and $r\leq 2^i+i-2$ so that, $2(r-i)+1 \leq 2^{i+1}-3$, a contradiction. 

%\textcolor{blue}{I do not understand the last two sentences from ``Therefore, ....''. The distinct colors in $C_0$ are at least $4i-2$ (not $4i$ as stated) for a total of at least $2^{i+1}-6$ (not $2^{i+1}-2$) colors. This is not in contradiction with the number of colors that is at most  $2^{i+1}-3$.  }\par
%\textcolor{green}{The error in computation is ratified.}

Now, if we do the same procedure with $2(r-i)+2$ colors, then $|\chi_0| \geq 4i - 4$ and $|\chi_j| \geq 2^{i-1}(4i - 4j - 6)$ so that $|\cup_{j=0}^{i-2}\chi_j| \geq 2^{i}+2^{i-1}-2$. When $r \leq 2^{i-1}+2^{i-2} + i - 3$, the number of colors $2(r-i)+2 \leq 2^{i}+2^{i-1}-4$, a contradiction. \par

Hence, the theorem.
\end{proof}

The general position variation of mutual-visibility chromatic number was introduced by \cite{MVC3} in 2024. The minimum number of colors required to color the vertices of $G$ such that each color class is in general position is called \textit{general position chromatic number}, denoted by $\chi_{\text{gp}}(G)$. The color classes obtained in the proof of Theorem \ref{binary} are also in general position except for $r=2^{i-1} + 2^{i-2} + i - 2$. Hence, we have the following corollary.

\begin{corollary}
    If $r \geq 2$ and $r \neq 2^{i-1} + 2^{i-2} + i - 2$ then $\chi_{gp}(GT(r))=\chi_{\mu}(GT(r))$.
\end{corollary}

The mutual-visibility coloring of glued binary trees can be naturally extended to glued $t$-ary trees. We use the notations $P_{i,j}$,$P_{i,j}'$, $Q_{i,j}$ and $Q_{i,j}'$ similar to the notation used for $GT(r)$. Now, Lemma \ref{lem} can be generalized to glued $t$-ary trees as follows.
\begin{lemma} \label{lemma}
	Let $S$ be a mutual-visibility set of $GT(r,t)$. Let $C_{n,k}$ denote the circle passing through the quasi-leaves $v_n$ and $v_k$. Then $|{S \cap V(C_{i,j}) }| \leq 3$.
\end{lemma}
        
	\begin{theorem}
		For $r, t \geq 2$,
		
		$$\chi_\mu(GT(r,t))=
		\begin{cases}
			2(r-i)+3,~ ~ \frac{t^{i-1}-1}{t-1}+i-1 \leq r \leq \frac{t^{i-1}-1}{t-1}+\frac{t^{i-1}}{2}+i-2\\
            2(r-i)+2,~ ~ \frac{t^{i-1}-1}{t-1} +\frac{t^{i-1}}{2}+i-1 \leq r \leq \frac{t^i-1}{t-1}+i-1
			
		\end{cases}$$
		\end{theorem} 
        
	\begin{proof}
	We begin the proof by establishing a coloring of $GT(r,t)$ using the number of colors specified in the theorem.\par
	
	\ni{\bf Case 1: $\frac{t^{i-1}-1}{t-1} +\frac{t^{i-1}}{2}+i-1 \leq r \leq \frac{t^i-1}{t-1}+i-1$}.\\
    We proceed with the same coloring pattern used in Theorem 4.6 so that we reach a stage where the vertices of depth $i-1$ in $V_r^{(2)}$, $v'_{i,1},v'_{i,2},\ldots,v'_{i,j}$ will be already colored, where $j=r-i-(\frac{t^{i-1}-1}{t-1}) \geq \frac{t^{i-1}}{2} - 1$, (since, $r \geq \frac{t^{i-1}-1}{t-1} +\frac{t^{i-1}}{2}+i-1 $). Similarly, in $V_r^{(2)}$, $v'_{i,a},v'_{i,a-1},\ldots,v'_{i,a-j-1}$ will be already colored, where $a=t^{i-1}$ and $j \geq \frac{t^{i-1}}{2}-1$. When $r \geq \frac{t^{i-1}-1}{t-1} +\frac{t^{i-1}}{2}+i-1$, we have $j \geq \frac{t^{i-1}}{2}$ so that there are fewer than $\frac{t^{i-1}}{2}$ uncolored vertices of depth $i-1$. In this case, color all the remaining uncolored vertices with color $2(r-i)+2$. When $r = \frac{t^{i-1}-1}{t-1} +\frac{t^{i-1}}{2}+i-1$, color $v_{i-1,j+1}$ with color 1 and all the remaining vertices with color $2(r-i)+2$ to obtain a mutual-visibility coloring of $GT(r,t)$ with $2(r-i)+2$ colors.
	
	\ni{\bf Case 2: $ \frac{t^{i-1}-1}{t-1}+i-1 \leq r \leq \frac{t^{i-1}-1}{t-1}+\frac{t^{i-1}}{2}+i-2$}.\\
		\ni{\bf Sub case 1: $r= \frac{t^{i-1}-1}{t-1}+i-1 $}.\\
	In this case, we proceed with the same coloring pattern given in Case 1 up to $k=r-i+1$ so that the entire vertices of $GT(r,t)$ are colored using  $2(r-i)+3$ colors.

\ni{\bf Sub case 2: $\frac{t^{i-1}-1}{t-1}+i \leq r \leq \frac{t^{i-1}-1}{t-1}+\frac{t^{i-1}}{2}+i-2$}.\\
	The coloring procedure is the same as Case 1, where the only difference is that more than half of the vertices of depth $i-1$( all vertices when $r=\frac{t^{i-1}-1}{t-1}+i$) will be uncolored. We can use one color each for these uncolored vertices in $V_r^{(1)}$ and $V_r^{(2)}$ so as to obtain a mutual-visibility coloring of $GT(r,t)$. 
	
	To prove the converse, first we prove that $2(r-i)+1$ colors are not enough for a mutual-visibility coloring of $GT(r,t)$, where $i$ is the unique integer such that $\frac{t^{i-1}-1}{t-1}+i-1 \leq r \leq \frac{t^i-1}{t-1}+i-1$. Assume the contrary that there exists a mutual-visibility coloring with $2(r-i)+1$ colors.\par
	
	Consider the cycle $C_{1,t^r}$. By Lemma \ref{lemma}, a color can be used at most 3 times to color the vertices of this cycle. Since we have only $2(r-i)+1$ colors, and there are $4r$ vertices in $C_{1,t^r}$, at least $4r - 2[2(r-i)+1] = 4i-2$ colors must be used at least three times. Let $\chi_0$ be the collection of colors which are used three times to color the vertices of $C_{1,t^r}$ so that $|\chi_0|\geq4i-2$. Now, consider $(t-1)$ cycles $C_{mt^{r-1},mt^{r-1}+1}$ where $m \in [t-1]$, note that all the cycles have length $4r$. Let $S_m=V(C_{mt^{r-1},mt^{r-1}+1})\setminus \{V(C_{1,t^r}) \cup V(C_{(m-1)t^{r-1},(m-1)t^{r-1}+1} )\}$, $m \in [t-1]$ . Note that there are $4r-4$ vertices in $S_m$, if $m \in [t-2]$ and $4r-6$, if $m=t-1$. Similar to the discussion above, at least $4i - 6$ colors must be used three times to color those vertices in $S_m$, $m\in [t-2]$ and at least $4i-8$ colors must be used three times to color vertices in $S_{t-1}$. Let $\chi_m$ denotes the collection of colors which are used three times to color the vertices in $S_m$ , $ m \in [t-1]$ so that $|\chi_m|\geq4i - 6$, where $m \in [t-2]$ and $|\chi_m|\geq4i - 8$, when $m =t-1$. \\
\ni{\textbf{Claim 1:}} $\chi_n \cap \chi_m = \phi$, for every $n \neq m$.\\
On the contrary, if $c \in \chi_m \cap \chi_n$ then $c$ must be used at least two times in the geodesic path $v_{1,1} - v_{mt^{r-1}}-v'_{1,1}$ or in the geodesic path $v_{1,1}-v_{mt^{r-1}+1}-v'_{1,1}$. Similarly, $c$ must be used at least two times in the geodesic $v_{3,tn}-v_{nt^{r-1}}-v'_{3,tn}$ or $v_{3,tn+1}-v_{nt^{r-1}+1}-v'_{3,tn+1}$. But then, in one of the cycles $C_{mt^{r-1},nt^{r-1}}, C_{(m+1)t^{r-1},nt^{r-1}}, C_{mt^{r-1},(n+1)t^{r-1}}, C_{(m+1)t^{r-1},(n+1)t^{r-1}} $ , $c$ will be used four times which contradicts Lemma \ref{lem}. Therefore, $\chi_n \cap \chi_m= \phi$.\\
Generally in $n^ {th} $ step we consider $(t-1)t^{n-1}$ cycles $C_{mt^{r-n},mt^{r-n}+1}$ where $m \in [t^n] - t\mathbb{N}$. Note that in the $n^{th}$ step all the cycles have length $4(r-n+1)$. Set $S^0_m=V(C_{1,2^r})$ and define  $S_m^{n}=V(C_{mt^{r-n},mt^{r-n}+1}) \setminus  \{V( C_{1,t^r}) \cup V( C_{(m-1)t^{r-n},(m-1)t^{r-n}+1})  \cup S_m ^{n-1} \}$, for every $n \in [i-1]$ . Note that there are $4r-4n$ vertices in $S^n_m$,  if $ m \not \equiv  1\pmod t$ and $4r-4n-2$, if $ m \equiv 1 \pmod t$. Similar to the discussion above, since we have only $2(r-i)+1$ colors, at least $4r-4n - 2[2(r-i)+1] = 4i-4n-2$ colors must be used at least three times to color those vertices in  $S_m^n$, where $ m \not \equiv 1 \pmod t$, and at least $4i-4n-4$ colors must be used three times to color those vertices in $S_m^n$, where $ m \equiv 1 \pmod t$. Let $\chi^n_m$ denotes the collection of colors which are used three times to color the vertices in $S_m^n$, $ m \in [t-1] - t\mathbb{N}$ so that $|\chi^n_m|\geq 4i - 4n-2$, if $ m \not \equiv 1 \pmod t$ and $|\chi^n_m|\geq 4i - 4n-4$, if $ m \equiv 1 \pmod t$. \par
Now, we proceed to prove a general case of Claim 1.\par
\ni{\textbf{Claim 2:}}  $\chi^n_m$ are disjoint for every $m \in [t-1]-t\mathbb{N}$ and $n \in [i-1]$.\\
	 On the contrary, if $c \in \chi^n_m \cap \chi^q_p$, then $c$ must be used at least two times in the geodesic path $v_{(n-2),m}-v_{mt^{r-n}}-v'_{(n-2),m}$ or in the geodesic path $v_{(n-2),(m+1)}-v_{mt^{r-n}+1}-v'_{(n-2),(m+1)}$. Similarly, $c$ must be used at least two times in the geodesic path $v_{(q-2),p}-v_{pt^{r-q}}-v'_{(q-2),p}$ or in the geodesic path $v_{(q-2),(p+1)}-v_{pt^{r-q}+1}-v'_{(q-2),(p+1)}$. But then $c$ will be used at least four times in one of the cycles $C_{mt^{r-n},pt^{r-q}}$, $C_{(m+1)t^{r-n},pt^{r-q}}$, $C_{mt^{r-n},(p+1)t^{r-q}}$ and $C_{(m+1)t^{r-n},(p+1)t^{r-q}}$, which contradicts Lemma \ref{lemma}. Therefore, $\chi^n_m \cap \chi^q_p = \phi$.\par
Note that we have considered $1+(t-1)+t(t-1)+t^2(t-1)+\ldots+t^{i-2}(t-1) = 1+ t^{i-1} - 1 = t^{i-1}$ cycles. Since all the $t^n-1$ cycles $\chi^n_m$ are disjoint and since cardinality of  elements in $\chi^n_m$ is at least $4i-4n-2$, if  $ m \not \equiv 1 \pmod t$ and $4i-4n-4$, if  $ m \equiv 1 \pmod t$ implies that $|\bigcup_{m,n} \chi^n_m| \geq \sum_{n=1}^{i-1}\Bigl ( (t-2) t^{n-1}(4i-4n-2) + t^{n-1}(4i-4n-4) \Bigr)= \frac{2(t^{i}-1)}{t-1}$. Therefore, at least $\frac{2(t^{i}-1)}{t-1}$ colors are used three times to color $GT(r,t)$. But, we have only $2(r-i)+1 $ colors and $r \leq  \frac{t^i-1}{t-1}+i-1$ so that $2(r-i)+1 \leq \frac{2(t^{i}-1)}{t-1}-1$, a contradiction. \\
Therefore, at least $2(r-i)+2$ colors are required for a mutual-visibility coloring of $GT(r,t)$. Now, let us investigate whether $2(r-i)+2$ colors are sufficient.		
Suppose we have $2(r-i)+2$ colors. Consider the cycles $C_m^n$, $n \in [i-2]$. Then by similar arguments, all these cycles  are disjoint and  cardinality of  elements in $\chi^n_m$ is at least $4i-4n-4$, if  $ m \not \equiv 1 \pmod t$ and $4i-4n-6$, if  $ m \equiv 1 \pmod t$, implies that $|\bigcup_{n,m} \chi^m_n| \geq  \frac{2(t^{i-1}-1)}{t-1} + 2t^{i-2}$, $n\in [i-2]$. Let $S$ be the collection of all vertices of depth $(i-1)$ of  $T_r^{(1)}$ that are not part of any $C_m^n$ for $n\in [i-2]$. Note that  $|S|=t^{i-1}-2t^{i-2}$ and $|S| >0$ for $t\geq 3$. Let $T$ be the collection of quasi-leaves obtained by taking one quasi-leaf of the binary tree rooted at $k$, for each $k \in S$. Now consider the cycles $C_{k,l}$ where $k,l \in T$. In each of these cycles there are exactly $2(r-i+3)$ vertices that are not part of any $C_m^n$, $n \in [i-2]$. Since we have only $2(r - i) + 2$ colors, at least two colors must be repeated at least three times in every cycle, and these colors are distinct. Using similar arguments, we can prove that these colors are also distinct from $\chi_m^n$, $n \in [i-2]$. Therefore, at least $\frac{2(t^{i-1}-1)}{t-1} -2t^{i-2} + t^{i-1}-2t^{i-2}=\frac{2(t^{i-1}-1)}{t-1} +t^{i-1}$ distinct colors must be used to color $GT(r,t)$.  But in this case, if  $r \leq \frac{t^{i-1}-1}{t-1}+\frac{t^{i-1}}{2}+i-2$, so that $2(r-i)+2 \leq \frac{2(t^{i-1}-1)}{t-1}+t^{i-1}-2$.
Hence, the theorem.
\end{proof}
    
\section{Concluding Remarks}	    
In this paper we prove the NP-completeness of the mutual-visibility coloring problem, which was left as an open question in \cite{MVC1}, reinforcing the theoretical challenge in determining the mutual-visibility chromatic number of arbitrary graphs and providing strong motivation to find the exact value and tight bounds for specific graph classes. We also determine the exact value of glued binary trees and glued $t$-ary trees. Several natural directions remains open for future investigation. In particular finding the mutual-visibility chromatic number of  highly symmetric or self similar structures such as the hypercubes, hypercube-like graphs, the Hamming graphs and the generalized  Sierpi\'{n}ski graphs, including Sierpi\'{n}ski triangle graphs and Sierpi\'{n}ski gasket graphs are interesting.\par
 
\section{Acknowledgments}
Saneesh Babu acknowledges Kerala State Council for Science,Technology and Environment (KSCSTE) for the financial support. For this work, Gabriele Di Stefano has been partially supported by the Academic Research Project ``MoNet'' at the University of L'Aquila.

\end{document}